\documentclass[11pt]{amsart}

\usepackage{amssymb,amsfonts}
\usepackage[all,arc]{xy}
\usepackage{enumerate}
\usepackage{mathrsfs}
\usepackage{eufrak}
\usepackage{setspace}

\newtheorem{theorem}{Theorem}[section]
\newtheorem{proposition}[theorem]{Proposition}
\newtheorem{lemma}[theorem]{Lemma}
\newtheorem{corollary}[theorem]{Corollary}

\theoremstyle{remark}
\newtheorem{remark}[theorem]{Remark}

\renewenvironment{proof}{{\noindent\bf Proof.}}{\hfill $\Box$\par\vskip3mm}

\newcommand{\Hom}{{\rm Hom}}

\newcommand{\Aa}{\mathcal{A}}

\newcommand{\Cc}{\mathcal{C}}

\newcommand{\Ee}{\mathcal{E}}

\newcommand{\Gg}{\mathcal{G}}
\newcommand{\Hh}{\mathcal{H}}

\newcommand{\Mm}{\mathcal{M}}

\def\KK{{\mathbb K}}

\theoremstyle{definition}

\theoremstyle{remark}

\makeatletter
\let\c@equation\c@thm
\makeatother
\numberwithin{equation}{section}

\title{Generators for Comonoids and Universal Constructions}
\author{ADNAN H. ABDULWAHID AND\\
  MIODRAG C. IOVANOV}
\date{}

\begin{document}

\begin{abstract}
\noindent We investigate cofree coalgebras, and limits and colimits of coalgebras in some abelian monoidal categories of interest, such as bimodules over a ring, and modules and comodules over a bialgebra or Hopf algebra. We find concrete generators for the categories of coalgebras in these monoidal categories, and explicitly construct cofree coalgebras, products and limits of coalgebras in each case. This answers an open question in \cite{Agore} on the existence of a cofree coring, and constructs the cofree (co)module coalgebra on a $B$-(co)module, for a bialgebra $B$. 
%Let $R$ be an arbitrary ring with $1$ and $A$ be an $R$-algebra. Let $V$ be an $A$-bimodule.  The question is does $V$ have an $A$-cofree coring? Agore, A. L. mentioned this question in [1] as an open problem. We are going to show that the $A$-cofree coring exists over an arbitrary $A$-bimodule $V$. More interestingly, we are going to construct   that $A$-cofree coring.  
\end{abstract}

%, 20N99, 05E10} 
%Primary 16W30; Secondary 16S90, 16Lxx, 16Nxx, 18E40}
%\thanks{$^*$}
\date{}

\maketitle

% generates the title

\noindent

\section{Introduction}

The free algebra on a set or vector space, i.e. non-commutative polynomial algebras, as well as the free group or monoid, represent fundamental constructions in algebra, and their importance is already well established. Moreover, the construction of the free algebra - the tensor algebra - is suitable for use in monoidal categories. We refer to \cite{Bakalov}, \cite{Mac Lane} for basics on monoidal categories. A monoid in the monoidal category $(\Cc,\otimes,I)$ (or algebra when $\Cc$ is abelian) is a triple $(A,m,u)$, where $m:A\otimes A\rightarrow A$ and $u:I\rightarrow A$ are the multiplication and unit respectively, satisfying appropriate associativity and unital conditions which are the categorical analogue of the usual monoid/algebra axioms; morphisms of monoids are defined by analogy too. Let $Mon(\Cc)$ denote the category of monoids in $\Cc$, and $U:Mon(\Cc)\rightarrow \Cc$ be the forgetful functor. By a basic observation (see \cite[Chapter 7]{Mac Lane}), if $(\Cc,\otimes,I)$ has denumerable coproducts and $\otimes$ commutes with these coproducts, then the tensor algebra $T(V)=\bigoplus\limits_{n\geq 0}V^{\otimes n}$ gives a left adjoint to $U$ (the free algebra). 

%\vspace{.1cm}

A comonoid in $\Cc$ is simply a monoid in the opposite category $\Cc^0$. Denote $CoMon(\Cc)=Mon(\Cc^0)$. Both monoids and comonoids are important objects; hence, finding a universal ``comonoid" is a basic question of interest. Specifically, the problem formulates as follows: given a monoidal category $\Cc$, does the forgetful functor $Mon(\Cc^0)\rightarrow \Cc^0$ have a left adjoint, or equivalently, does the forgetful $CoMon(\Cc)\rightarrow \Cc$ have a right adjoint, the ``cofree" comonoid? 
% A free monoid in $\Cc^0$ Specifically, this formulates into finding a free monoid functor in $\Cc^0$, or cofree comonoid on an object $V$ in $\Cc$. Equivalently, the problem is: when does the forgetful functor $CoMon(\Cc)\rightarrow \Cc$ have a right adjoint? 
This is often of interest for the same types of categories in which one considers monoids and to which Mac Lane's construction of the free monoid applies. However, to use it to obtain cofree comonoids in $\Cc$, equivalently, for monoids in $\Cc^0$, one needs that $\otimes$ commutes with denumerable products in $\Cc$, and this is almost never the case for any monoidal category of interest; this fails, for example, even in the case of vector spaces.

\vspace{.2cm}

Hence, one needs different considerations to construct cofree comonoids, and, of course, the natural method is to use Freyd's adjoint functor theorem. We will consider this question for several categories that present interest, namely, the category ${}_A\Mm_A$ of bimodules over a ring $A$, and the categories of modules ${}_B\Mm$ or comodules $\Mm^B$ over a bialgebra (or Hopf algebra) $B$. We will denote $Alg(\Cc)=Mon(\Cc)$ and $CoAlg(\Cc)=CoMon(\Cc)$ when $\Cc$ is abelian, i.e algebras=monoids and coalgebras=comonoids as it is the usual terminology in this case. We recall that a coalgebra in the category of bimodules over a ring $A$ is called an $A$-coring. Corings have been extensively studied over the past 15 years, especially due to the fact that they are very common in algebra and representation theory. Many important categories can be realized as comodules over suitable corings: modules, categories of actions and co-actions of bialgebras, modules over algebras in certain monoidal categories, chain complexes, (differentially) graded modules, quasi-coherent sheaves, generalized quiver representations, etc.; we refer the reader to \cite{Brzezinski} and references therein. The existence of cofree coalgebras over fields has been known for a long time \cite{Block}, \cite{Dascalescu}; it was obtained for coalgebras over arbitrary commutative rings in \cite{Barr}, and some constructions for the non-coassociative case are given in \cite{Fox}. The cofree coring over $A^n$ is constructed in \cite{Hazewinkel}, and the problem is answered for von Neumann regular rings in \cite{Porst2} (see also \cite{Porst1}). The main difficulty in generalizing to arbitrary bimodules (\cite{Fox}, \cite{Porst2}) is to deal with the pathological behavior of the tensor product over an arbitrary ring $A$. The question of whether a cofree coring always exists is proposed as an open problem in \cite{Agore}. A closely related question turns out to be whether the category of comonoids in a monoidal category is complete, and in particular, if limits of $A$-corings exist (see \cite{Agore}). 

\vspace{.2cm}

A natural setting for the study of such questions is, of course, the general categorical framework of the Special Adjoint Functor Theorem [SAFT];
%These questions have a natural answer in a general categorical framework using the Special Adjoint Functor Theorem [SAFT]; 
(see also \cite{Porst2}). We investigate cofree coalgebras and completeness for the category of coalgebras $CoAlg(\Cc)$ for the case when $\Cc={}_A\Mm_A$, or $\Cc={}_B\Mm$ or $\Cc=\Mm^B$, which are of particular interest among monoidal categories. In general, it is easy to note that, when $\Cc$ is co-wellpowered and cocomplete, then $CoAlg(\Cc)$ is cocomplete, co-wellpowered and the forgetful preserves colimits. Although existence of the adjoints can be obtained in greater generality, our emphasis is on finding concrete constructions of these universal objects, which can be written out in terms of generators of the respective category of coalgebras. Hence, our main goal is to give explicit sets of generators for each of the above categories, and obtain these adjoints explicitly. These results can also be regarded as counterparts of the fundamental (local) finiteness theorem of coalgebras over fields. First, in Section \ref{s.p} we recall the construction of the adjoint in SAFT, and give a simple construction using the same idea as the proof of SAFT, that explicitly describes the products and limits in the category of comonoids $CoMon(\Cc)$ in a suitable monoidal category $\Cc$. In Section \ref{s.1} we note that $A$-corings are generated by corings of infinite cardinality less or equal to that of $A$ and $\aleph_0$; over von-Neumann regular rings, or countably generated algebras, countably generated or countable dimensional corings generate the category. In Section \ref{s.3} we show that if $B$ is a bialgebra, $B$-module coalgebras which are finitely generated over $B$ form a set of generators for the category of all $B$-module coalgebras; we also show that finite dimensional $B$-comodule coalgebras generate the category $CoAlg(\Mm^B)$. When $B$ is a Hopf algebra, comatrix coalgebras in $\Mm^B$ can be defined and they generate $CoAlg(\Mm^B)$. We then apply the formulas obtained in Section \ref{s.p} to concretely describe cofree coalgebras and limits in each case. Our emphasis is on the constructive direct approach, which yields working formulas for these universal objects; these can useful to the general reader working with concrete monoidal categories. %With a more general algebra reader in mind, we 

\section{Preliminaries}\label{s.p}

In this section, we recall the construction used in the Special Adjoint Functor Theorem (or SAFT) due to P. Freyd, and use it to provide formulas that will be applied in several situations. A functor $U$ that has a right adjoint commutes with limits, and the SAFT says that under additional natural hypothesis, this is enough to find such an adjoint. Recall that a category $\Cc$ is called co-wellpowered if for every object $C$ in $\Cc$, the equivalence classes of epimorphisms with source $C$ form set. The (dual) SAFT asserts that {\it if $\mathcal{C}$ is a cocomplete (has small colimits), co-wellpowered category and with a generating set, then every co-continuous functor $U$ (i.e. $U$ commutes with colimits) from $\mathcal{C}$ to a locally small category $\Mm$ has a right adjoint}.

%In general, if a functor $U$ has a right adjoint $R$ then $U$ commutes with colimits. In a way, the SAFT asserts that under additional natural hypotheses, this commutativity is enough for finding right adjoints. For convenience, recall this well known result of P. Freyd:  here. Recall first that a category $\Cc$ is called co-wellpowered if for every object $C$ in $\Cc$, the equivalence classes of epimorphisms with source $C$ form set. $\Cc$ is cocomplete if it has all small colimits.

%\begin{theorem} [SAFT] 
%If $\mathcal{C}$ is a cocomplete, co-wellpowered category and with a generating set, then every co-continuous functor $U$ (i.e. $U$ commutes with colimits) from $\mathcal{C}$ to a locally small category $\Mm$ has a right adjoint.
%\end{theorem}

We also briefly recall the construction of the adjoint, since it will be of importance in our examples. It can be formulated in the following way. Let $\Gg$ be a set of generators of $\Cc$, and for each $M$ in $\Mm$, let $\overline{\Ee}$ be the category whose objects are pairs $(C,f)$ with $C$ in $\Cc$, and $f:U(C)\rightarrow M$ is a morphism in $\Mm$. Morphisms in this category are $(\varphi,*):(C,f)\rightarrow (C',f')$ such that $\varphi:C\rightarrow C'$ is a morphism in $\Cc$, and $*$ means simply that the condition $f'U(\phi)=f$ is satisfied. Consider also $\Ee$ the full subcategory of $\overline{\Ee}$ of objects of the type $(G,f)$ for $G\in \Gg$. Let $F:\Ee\rightarrow \Cc$ be defined by $F(G,f)=G$ and $F(\varphi,*)=\varphi$ on morphisms. Then $F$ is a functor, and since the class of objects of $\Ee$ is a set (since $\Cc$ is co-wellpowered), it has a colimit. Let $(R(M),f_M)=\lim\limits_{\stackrel{\longrightarrow}{(G,f)\in\Ee}}F(G)=\lim\limits_{\longrightarrow}F$. Then $(R(M),f_M)$ is a final object in $\Ee$, which in turn helps proving that it is a final object in  $\overline{\Ee}$. We need to recall this argument, as it will be used later. Given $C$ in $\Cc$ together with a morphism $f:U(C)\rightarrow M$ in $\Mm$, use the generators $\Gg$ to write $C=\lim\limits_{\stackrel{\longrightarrow}{p:G\rightarrow C, G\in \Gg}}G$ and 
$(C,f)=\lim\limits_{\stackrel{\longrightarrow}{p:G\rightarrow C}}(G,fU(p)).$  
Since $(R(M),f_M)=\lim\limits_{\stackrel{\longrightarrow}{h:U(G)\rightarrow M}}(G,h)$, and there is a canonical map between the two inductive systems induced by the correspondence $(G,p)\longmapsto (G,fU(p))$,% $(G,[p:G\rightarrow C])\longmapsto (G,[fU(p):U(G)\rightarrow M])$, 
we get a canonical morphism $\psi_C:C\rightarrow R(M)$ in $\Cc$ which turns out unique with $f_M\psi_C=f$.

%Then, there is a unique canonical morphism $\psi_C:C\rightarrow R(M)$ in $\Cc$ with $f_M\psi_C=f$. Indeed, if $\Ee'=\bullet \rightarrow C=\{(G,p)|p:G\rightarrow C, G\in\Gg, p\in \Hom_\Cc(G,C) \}$ is the comma category with the appropriate morphisms and $H:\Ee'\rightarrow \Ee$ is the functor defined by $H(G,p)=(G,fp)$, then the colimit of this functor  produces the morphism $\psi_C:C\rightarrow R(M)$ (i.e. $\psi_C$ is the natural map between the two resulting colimits).

We will need the following short hand formulas to explicitly compute several functors. %Let us write for short 
\begin{eqnarray}
R(M)& = & \lim\limits_{\stackrel{\longrightarrow}{h: U(G)\rightarrow M|G\in\Gg}}G \label{e.co1}\\
%\end{eqnarray}
%and more precisely 
%\begin{eqnarray}
(R(M),f_M) & = & \lim\limits_{\stackrel{\longrightarrow}{h: U(G)\rightarrow M|G\in\Gg}}(G,h:U(G)\rightarrow M)\label{e.co2}
\end{eqnarray}

We will also need the following observation, which is similar to the above discussion. Let $\Mm,\Cc$ be categories, $U:\Cc\rightarrow \Mm$ be a faithful functor. For convenience, we will interpret $U$ as a ``forgetful" functor and slightly abuse notation sometimes to write $U(C)=C$; similarly, if $f:U(C)\rightarrow U(C')$ is a morphism in $\Mm$, we say $f$ is a morphism in $\Cc$ if $f=U(g)$ for some (unique) $g\in\Hom_\Cc(C,C')$. %if $f:U(C)\rightarrow U(C')$ is a morphism in $\Mm$, we say $f$ is a morphism in $\Cc$ if $f=U(g)$; we may do this since such a $g$ is unique if it exists. We will also slightly abuse notation and write $U(C)=C$ for a $C\in \Cc$ (i.e. regard $C\in\Cc$ as an object of $\Mm$ via the ``forgetful" $U$). 
Fix an object $N$ in $\Mm$ and objects $(C_i)_{i\in I}$ in $\Cc$ and a family of morphisms $f_i:N\rightarrow U(C_i)$ in $\Mm$; the set $I$ can be empty. Assume $\Cc$ and $\Mm$ each have an initial object denoted $0$ and $U(0)=0$. Let $\Hh$ be the comma category whose objects are $(E,p), E\in\Cc, p\in \Hom_\Mm(U(E),N)$ and such that $f_ip$ is a morphism in $\Cc$ for all $i$; the morphisms $(E,p)\rightarrow (E',p')$ of this category are given by $h\in \Hom_\Cc(E,E')$ with $p'h=p$ (i.e. $p'U(h)=p$). We show that this category $\Hh$ has a final object under the same  conditions as in SAFT; when $I=\emptyset$, this is exactly the construction in SAFT.

\begin{proposition}\label{p.cat}
If $\Cc$ is cocomplete, co-wellpowered and has a set $\Gg$ of generators, and $\Mm,\Cc$ have the same initial object $0$. Then the category $\Hh$ has a final object, which we denote by $(\Hh_0(N,(f_i)_i),p_0)$.% (depends on the fixed system $G$ and $\Hh$).
\end{proposition}
\begin{proof}
Let $\Hh_0$ be the full subcategory of $\Hh$ whose objects $(G,p)$ are such that $G\in\Gg$. Note that $\Hh_0$ has at least one object $0\in \Cc$, since $0$ is also the initial object in $\Mm$, so the compositions $0\rightarrow N\stackrel{f_i}{\rightarrow}C_i$ are morphisms in $\Cc$. Let
\begin{eqnarray}
E_0=\lim\limits_{\stackrel{\longrightarrow}{(G,p)\in\Hh_0}}G\label{eq.1}
\end{eqnarray}
be a colimit in $\Cc$, which is also a colimit in $\Mm$ (i.e. $U(E_0)=\lim\limits_{\longrightarrow}U(G)$), and let $p_0:E_0\rightarrow N$ be the canonical morphism in $\Mm$ obtained from this colimit. Let $\sigma_{(G,p)}:G\rightarrow E_0$ be the canonical morphisms (in $\Cc$) of the colimit, so $p_0\sigma_{(G,p)}=p$ (more precisely, $p_0U(\sigma_{(G,p)})=p$). 
$$\xymatrix{
G\ar[r]^{\sigma_{(G,p)}}\ar[dr]_{p} & E_0\ar[d]^{p_0} \ar[dr]^{f_ip_0} \\
& N \ar[r]_{f_i} & C_i 
}$$
We show that $p_0$ is a morphism in $\Hh$. Fix $i$, and consider the morphisms $f_ip:G\rightarrow C_i$ which are morphisms in $\Cc$ (since $(G,p)$ is in $\Hh$; that is, $f_ip=U(h_i)$); there is a unique morphism in  $\theta:E_0\rightarrow C_i$ in $\Cc$ such that $\theta\sigma_{(G,p)}=f_ip$ for all $(G,p)$. But $f_ip_0$ (a morphism in $\Mm$) already satisfies this equation, and since the colimit in \ref{eq.1} is also a colimit in $\Mm$ by hypothesis, we obtain that $\theta=f_ip_0$ (obviously, here we implicitly use that $U$ is faithful). Hence, $f_ip_0$ is a morphism in $\Cc$. This shows that $p_0$ is in $Mor(\Hh)$. 
%First, the limit in \ref{eq.1} is also a limit in $\Mm$ by hypothesis; so, if we f
The fact that $(E_0,p_0)$ is the final object in the category $\Hh$ follows as in the proof of SAFT, along the lines noted in the comments preceding this proposition.
\end{proof}

%In the context of the above proposition, we will denote $(\Hh_0(N;(f_i)_i),p_0)$ the initial object of this category. This standard proposition has several consequences. Of course, if the set $I$ is chosen empty, the above proposition leads to the construction of the adjoint functor in SAFT. We can also use it to construct direct products in the subcategory $\Mm$ satisfying such suitable conditions.

\begin{remark}\label{r.cat}
We will use the above notation $(\Hh_0(N;(f_i)_i),p_0)$ and proposition to construct limits. Let $\Mm,\Cc$ be a cocomplete categories, $U:\Cc\rightarrow \Mm$ a faithful functor preserving colimits. Assume $\Mm$ is also complete, $\Cc$ is co-wellpowered, has a set of generators $\Gg$ and $\Cc$ and $\Mm$ have initial objects $0$ with $U(0)=0$. Then $\Cc$ is complete (i.e. has all small limits), which can be computed as follows.\\
Let $F:\Aa\rightarrow \Cc$ be a functor from a small category to $\Cc$, and let $P=\lim\limits_{\longleftarrow}UF$ be the limit (in $\Mm$) of $UF$, with canonical morphisms $q_a:P\rightarrow UF(a)$ in $\Mm$, $a\in\Aa$, i.e. we may write $q_a:P\rightarrow F(a)$ are morphisms in $\Mm$. Let $(E_0,p_0)=(\Hh_0(P,(q_a)_a),p_0)$, $p_0:E_0\rightarrow P$ in $\Mm$ (more precisely $p_0:U(E_0)\rightarrow P$). %, but we again ommit the functor $U$ interpreted as ``forgetful"). 
We note that $(E_0,(\pi_a)_a)$ with $\pi_a=q_ap_0$ is the limit of $F$. Of course, $\pi_a$ are also morphisms in $\Cc$ by construction. To see $E_0$ is this limit, consider $h_a:C\rightarrow F(a)$ a family of morphisms in $\Cc$ which is a cone on $F$ (so if $x:a\rightarrow b$ then $h_bF(x)=h_a$).
$$\xymatrix{
& C\ar@{..>}[dl]_{s} \ar@{..>}[d]_{h} \ar[dr]^{h_a} &  \\
E_0\ar[r]_{p_0} & P\ar[r]_{q_a} & F(a)
}$$
 Then $U(h_a)\in Mor(\Mm)$ and $U(h_a)$ is a cone on $UF$, and by applying the limit property we get a unique morphism $h$ in $\Mm$ with $q_ah=h_a$. Since $q_ah=h_a$ are in fact morphism in $\Cc$ (images of $U$), there is a unique $s$ in $\Cc$ with $p_0s=h$, so $\pi_as=q_ah=h_a$. The uniqueness of $s$ with this property follows again easily by backward chase using the universal properties of $E_0$ and $P$. 
%Let $R$ be the right adjoint of $U$, and $f:R(P)\rightarrow P$ the canonical morphism in $\Mm$ given by the counit of this adjunction (constructed as before). Also, let $p_0:E_0\rightarrow R(P)$ be the map obtained by applying Proposition \ref{p.cat} to the family $q_af:R(P)\rightarrow F(a)$ of morphisms in $\Mm$ (with $P,F(a)$ objects of $\Cc$). We show that $(E_0,(\pi_a=q_a\circ f\circ p_0)_{a\in\Aa})$ is the limit of $F$. This is a now straightforward task: 
%consider $h_a:C\rightarrow F(a)$ a family of morphisms in $\Cc$ which is a cone on $F$ (so if $x:a\rightarrow b$ then $f_bF(x)=f_a$). Then $h_a\in Mor(\Mm)$ and $U(h_a)$ is a cone on $UF$, and by applying the limit property we get a unique morphism $h$ in $\Mm$ with $q_ah=h_a$. Let $r$ be the unique morphism in $\Cc$ with $fr=h$ (which exists by the universal property of $R$). Finally, note that $(q_af)r=q_ah=h_a$ are morphisms in $\Cc$, so by the universal property of $E_0$ given by Proposition \ref{p.cat}, there is a unique $s$ in $\Cc$ with $p_0s=r$.
%This shows that $\pi_a\circ s=h_a$. Moreover, such an $s$ is unique: if $\pi_a s'=h_a$, let $r'=p_0s'$ and $h'=fr'$. The morphism $h'\in Mor(\Mm)$ satisfies $q_ah'=q_afp_0s'=\pi_as'=h_a$ so $h'=h$ by the uniqueness of $h$; then $fr'=h'=h=fr$ so $r=r'$ ($h,h'$ are morphisms in $\Cc$). Finally, uniqueness of $s$ with $p_0s=r=r'=p_0s'$ implies $s=s'$.
\end{remark}

We note that a similar process is used in \cite{Agore} for describing products of coalgebras over a field, but with an extra-intermediate step: the canonical map $p_0:E_0\rightarrow P$ is obtained using the right adjoint $R$ of $U$ and the canonical map $c_P:R(P)\rightarrow P$; of course, our $p_0$ factors through $R(P)$, but this intermediate step is not necessary in general. %, and then applying the construction of the previous proposition to get a canonical map $h_P:E_0\rightarrow R(P)$, and $p_0=c_Ph_p$. Of course, it is obvious that the map $p_0$ we build factors through $R(P)$ by the universal property of $R(P)$, so this intermetiate step is not necessary in general.
We can also recall the following easy observation. 

\begin{proposition}\label{p.cowp}
Let $\Cc$ be a monoidal category, $CoMon(\Cc)$ be the category of comonoids of $\Cc$ and $U:CoMon(\Cc)\rightarrow \Cc$ be the forgetful functor. \\
(i) If $\Cc$ is cocomplete, then $CoMon(\Cc)$ is cocomplete and $U$ preserves colimits. \\
(ii) If furthermore $\Cc$ is co-wellpowered, then so is $CoMon(\Cc)$. 
\end{proposition}
\begin{proof}
(i) is well known and follows similar to the case of bimodules $\Cc={}_A\Mm_A$ \cite{Brzezinski}.\\
(ii) It is enough to show that if $\mu:C\rightarrow D$ and $\nu:C\rightarrow E$ are in $Epi(CoMon(\Cc))$ and equivalent as epimorphisms in $\Cc$, then they are equivalent in $Epi(CoMon(\Cc))$ (as epimorphisms). Let $h:D\rightarrow E$ be an isomorphism in $\Mm$ for which $h\mu=\nu$ and we show that $h$ is in fact an isomorphism in $\Cc$. %Consider the diagrams: 
%Let $C$ be a coring, and $Epi(C)$ the class of epimorhpisms in $Crg_A$ with domain $C$, up to the natural equivalence of epimorphisms, which is given by the isomorphisms in the comma category $C \downarrow Crg$. If $\mu:C\rightarrow D$ is an epi in $Crg_A$, then it is surjective. It is enough to show that if two epimorphisms $\mu:C\rightarrow D$ and $\nu:C\rightarrow E$ are equivalent as epimorphisms in $A$-bimodules, then they are equivalent as epimorphisms of corings. Let $h:D\rightarrow E$ be an isomorphism of $A$-bimodules for which $h\mu=\nu$. It suffices to show that $h$ is in fact an isomorphism of corings. Consider the diagrams: 
\[
\xymatrix{
C\ar@/^/[rr]^{\nu} \ar[r]|{\mu} \ar[d]^{\Delta_{\!_{C}}}  & D \ar[r]|{h} \ar[d]^{\Delta_{\!_{D}}} & E \ar[d]^{\Delta_{\!_{E}}} &  \\ 
C \otimes_{A} C \ar[r]^{\mu \otimes \mu} \ar@/_/[rr]_{\nu \otimes \nu} & D \otimes_{A} D \ar[r]^{h \otimes h} & E \otimes_{A} E }
\hspace{35pt}
\xymatrix{
C\ar@/^/[rr]^{\nu} \ar[r]|{\mu} \ar[rd]_{\epsilon_{\!_{C}}}  & D \ar[r]|{h} \ar[d]^{\epsilon_{\!_{D}}} & E \ar[ld]^{\epsilon_{\!_{E}}} \\ 
& A }
\]
The first and outside diagrams on the left picture are commutative since $\mu$ and $\nu$ are morphisms in $CoMon(\Cc)$, so $\Delta_{\!_{E}} h \mu = \Delta_{\!_{E}} \nu = (\nu \otimes \nu) \Delta_{\!_{C}} = (h \otimes h) (\mu \otimes \mu) \Delta_{\!_{C}} = (h \otimes h)\Delta_{\!_{D}} \mu $.
Since $\mu $ is an epi in $\Cc$, $\Delta_{\!_{E}} h = (h \otimes h)\Delta_{\!_{D}}$. Similarly, we have 
$\epsilon_{\!_{E}} h \mu = \epsilon_{\!_{E}} \nu = \epsilon_{\!_{C}}=\epsilon_{\!_{D}} \mu$, and it follows that $\epsilon_{\!_{E}} h = \epsilon_{\!_{D}}$. Hence, $h$ is a morphism in $CoMon(\Cc)$, and so is its $\Cc$ inverse.
\end{proof}

%again $U:\Cc\rightarrow \Mm$ be a functor, and let $\Mm=\Cc$

\section{Cofree corings and limits}\label{s.1}

Let $A$ be a ring and ${}_A\mathcal{M}_A$ the category of $A$-bimodules. Recall that an $A$-coring is a comonoid (or coalgebra) in the monoidal category ${}_A\Mm_A$ (see \cite{Brzezinski}). Let $Crg_A$ be the category of $A$-corings, and $U:Crg_A \rightarrow {}_A\mathcal{M}_A$ the forgetful functor. A subcoring $D$ of a coring $C$ is defined to be an $A$-sub-bimodule $D$ of $C$ such that the inclusion map $D\hookrightarrow C$ is a morphism of corings. Using the observations in the previous section and Proposition \ref{p.cowp}, we have

\begin{proposition}\label{p.coring1} 
$Crg_A$ is cocomplete, co-wellpowered and $U:Crg_A\rightarrow {}_A\Mm_A$ preserves colimits.
\end{proposition}

To explicitly find products of and co-free corings, we focus on constructing generators. We start with a simple motivational observation, which explains the main idea and recovers the result of \cite{Porst2}. A submodule $N$ of a left $A$-module $M$ is called a pure submodule provided that for any right $A$-module $P$, the induced map $P \otimes_{A} N \rightarrow P \otimes_{A} M$ is mono. %Pure right $A$-modules are defined analogously. 
Assume $A$ is von Neumann regular, so all $A$-modules are flat, and all monomorphisms are pure. If $C$ is an $A$-coring, and $S$ is an $A$-submodule of $C$, let $S'\supseteq S$ be the sub-bimodule generated by all $m,m_1,m_2$ for $m\in S$ and $\Delta(m)=\sum\limits_mm_1\otimes m_2$. Note that since $S'\subseteq C$ is pure as left and right $A$-module, we can regard $S'\otimes S'$ as a sub-bimodule of $C$, and so $\Delta(S)\subseteq S'\otimes S'$. Moreover, if $S$ is finitely generated as an $A$-bimodule, then $S'$ can be chosen  finitely generated too, since $\Delta$ is $A$-bilinear: let $S'$ be generated by the $m_1,m_2$'s for $m$ in a system of generators of $S$. We then iterate this construction to get $S^{(n+1)}=(S^{(n)})'$, and $S=\bigcup\limits_nS^{(n)}$. Regard $S\otimes S\subseteq C\otimes C$, and $\Delta(S)\subseteq S\otimes S$, since $\Delta(S^{(n)})\subseteq S^{(n+1)}\otimes S^{(n+1)}$. Using this, Remark \ref{r.cat} and Proposition \ref{p.coring1}, we get the following version of fundamental theorem for corings.
%Hence, we obtain the following proposition, which can also be regarded as a version of the fundamental theorem for corings; it follows from these observations, and Proposition \ref{p.coring1}:

\begin{proposition}
Let $A$ be a von Neumann regular ring, and $C\in Crg_A$. Then every finite subset of $C$ is contained in a subcoring $D$ of $C$ which is a pure left and right $A$-submodule and countably generated as an $A$-bimodule. Hence, $U$ has a right adjoint, and the category $Crg_A$ is complete.
\end{proposition}
%\begin{proof}
%(a) follows from the above observations; (b) follows from (a), since every coring will be the colimit of its subcorings which are at most countably generated as $A$-bimodules, and (c) follows from (b) using Remark \ref{r.cat}, by noting that the trivial module $0$ has a trivial coring structure and is an initial object in $Cor_A$ and ${}_A\Mm_A$.
%\end{proof}

By Remark \ref{r.cat}, the cofree coring on an $A$-bimodule $V$ is  $C(V)=\lim\limits_{\stackrel{\longrightarrow}{[f:C\rightarrow V]\in {}_A\Mm_A|{}_AC_A\,{\rm is\,}\aleph_0-{\rm\,generated}}}C.$ 
% with the colimit taken over all morphisms of bimodules from a coring $C$ which is countably generated as an $A$-bimodule. %Moreover, $Cor_A$ also has limits by Remark \ref{r.cat}.

Therefore, to find generators of $Crg_A$, we look for a type of fundamental theorem of corings, in which we deal with non-flatness in general. Naturally, the idea is to close up a sub-bimodule $S$ of $C\in Crg_A$ not only under $\Delta$ but also under ``purity". 
%Hence, this suggests that if one is to find generators of $Crg_A$, one can look for a fundamental type theorem as the above, in which non-flatness is resolved in some way. Given some sub-bimodule $S$ of a coring $C$, it involves closing $S$ to some $\overline{S}$ not only under $\Delta$, but also under ``purity", in such a way that the result obtained can be regarded as a subcorinng.
This is the essential idea of \cite{Barr}, which we briefly recall and note that it adapts straightforward to the most general case. The following theorem is well-known as Cohn's criterion for pure left $A$-modules, which we recall here for convenience \cite{Cohn}.

\begin{theorem} (P.M. Cohn) 
{\it Let $M$ be a left $A$-module, and $\bar{M} \subseteq M $ be a submodule. Then $\bar{M} $ is a pure submodule of $M$ if and only if every finite system of linear equations
\begin{eqnarray}
\sum\limits_{i=1}^{k_j} {\lambda_{ij}x_{i}} & = & \bar{m}_{j}, \hspace{15pt} j = 1,2,\dots,n \label{eq1}
\end{eqnarray}
with $\bar{m}_{j} \in \bar{M} $, $\lambda_{ij} \in A$ has a solution in $\bar{M} $ whenever it has one in $M$.}
\end{theorem}

\begin{remark}
The following is a well known procedure. Let $N$ be an $A$-sub-bimodule of $M$. Consider all finite systems (\ref{eq1}) for the left $A$-module $N$, and also for the right $A$-module $N$, and let $N'$ be the sub-bimodule of $M$ generated by all these solutions. Iterate this to get a chain $N\subseteq N'\subseteq N''\subseteq \dots N^{(n)}\subseteq \dots$, with union $N^*=\bigcup\limits_{n}N^{(n)}$. By Cohn's test  $N^*$ is a pure left and right sub-bimodule of $M$ containing $N$; moreover $|N^*|\leq \max\{|A|,|N|,\aleph_0\}$ ($|-|$ denotes cardinality).
%Then there is an $A$-sub-bimodule $N^*$ of $M$ such that $N\subseteq N^*\subseteq M $, $N^*$ is a pure both as a left and as a right submodule of $M$, and  $|N^*| \leq$ max $\{|N|,|A|,\aleph_0\}$, where $|X|$ denotes the cardinality of $X$.
%\end{proposition}
%\begin{proof}
%For each sub-bimodule $N$ of $M$, consider all finite system of linear (\ref{eq1}) for the left $A$-module $N$, and also for the right $A$-module $N$, and take $N'$ the sub-bimodule generated by all these solutions. Given $N$, inductively build up the chain $N\subseteq N'\subseteq N''\subseteq \dots \subseteq N^{(k)}\subseteq \dots$, $N^{(k+1)}=\left(N^{(k)}\right)'$; then $N^*=\bigcup\limits_{k}N^{(k)}$ satisfies the requirements.
\end{remark}

%To explain the construction, 
Given $A$-bimodules $M,N$ with sub-bimodules $M',N'$ denote $M'\cdot N'$ the image of the canonical map $M'\otimes N'\rightarrow M\otimes N$. If $(C,\Delta,\varepsilon)$ is a coring, we call a sub-bimodule $M$ of $C$ invariant if $\Delta(M)\subseteq M\cdot M$. As before, we note that every sub-bimodule $M$ of $C$ is contained in an invariant one: write $\Delta(m)=\sum\limits_{m}m_1\otimes m_2$ for each $m\in M$, and let $M'$ be the sub-bimodule generated by all $m_1$'s and $m_2$'s; it contains $M$ by the counit property: $m=\sum\limits_{m}\varepsilon(m_1)m_2$. Iterate this to get a sequence $M\subseteq M'\subseteq M''\subseteq \dots$; if $M^\sim$ is the union, it is clear that $M^\sim$ is invariant, and its cardinality again does not exceed $\max\{|A|,|M|,\aleph_0\}$.

Finally, let $S$ be a sub-bimodule of $C$, and iterate alternatively so that $S_0=S$, $S_{2n+1}=S_{2n}^\sim$, $S_{2n}=S_{2n-1}^*$. Let $D=\bigcup\limits_nS_n$. Then it is obvious that $D$ is invariant, and it is a left and right pure $A$-submodule of $C$. % (union of pure sub-bimodules, and of invariant sub-bimodules).  
Also, by purity, the canonical map $D\otimes D\rightarrow D\cdot D\subset C\otimes C$ is injective, which means $\Delta:C\rightarrow C\otimes C$ factors to $\Delta\vert_D:D\rightarrow D\otimes D$, and $D$ becomes a subcoring of $C$. Hence, we can conclude the main result of this section: 

\begin{theorem}
(i) Let $C$ be an $A$-coring. Then every $A$-sub-bimodule $S$ of $C$ is contained in a subcoring $D$ of $C$ which is pure as left and as right $A$-submodule, and such that $|D|\leq \max\{|S|,|A|,\aleph_0\}$.\\
(ii) $Crg_A$ is generated by all corings of cardinality $\leq \max\{|A|,\aleph_0\}$.\\
(iii) There is a cofree coring $C(V)$ on every $A$-bimodule $V$; moreover, using equations \ref{e.co1} and \ref{e.co2}
\begin{eqnarray*}
C(V) & = & \lim\limits_{\stackrel{\longrightarrow}{f: U(G)\rightarrow V|\,G\in Crg_A; \,|G|\leq \{|A|,\aleph_0\}}}G
\end{eqnarray*}\\
(iv) The category $Crg_A$ is complete. If $F:\Aa\rightarrow Crg_A$ a functor from a small category, let $(P,q_a:P\rightarrow C_a)=\lim\limits_{\stackrel{\longleftarrow}{a}}C_a=\lim\limits_{\stackrel{\longleftarrow}{\,}}UF$ be the limit of $A$-bimodules. The limit of $F$ is %, with canonical maps $q_a:P\rightarrow C_a$. Let $f:C(P)\rightarrow P$ be the cofree coring, and let $(C(P)_0,p_0)=(\Hh_0(P,(q_af)_a),p_0)$ be the construction of \ref{r.cat} relative to the morphisms $q_af$. In other words, 
$$C(P)_0=\lim\limits_{\stackrel{\longleftarrow}{[h:H\rightarrow C(P)]\,\,|\,\,h,\,q_afh\in Crg_A}}H$$ 
with canonical maps $(q_afp_0)_a$, where $(C(P)_0,p_0)=(\Hh_0(P,(q_af)_a),p_0)$ is the construction of  \ref{r.cat}.% is the colimit of the system $(C_a)_a$ (i.e. of the functor $F$).
\end{theorem}
We note that (iv) is also the main result of \cite{Porst1}, and (iii) was raised in \cite{Agore}; our emphasis is to answer these questions in a constructive way, and to explicitly produce the cofree coring and limit of corings.
%\begin{proof}
%(i) follows from the previous discussion, and (ii) is an obvious consequence of (i). (iii) follows from (ii) by applying SAFT, as $Crg_A$ is also co-well powered, has colimits which the forgetful functor preserves; (iv) follows from  \ref{r.cat} as before.
%\end{proof}
%We note that (iv) is also the main result of \cite{Porst1}, and (iii) was raised in \cite{Agore}; our emphasis is to answer these qeustion in a constructive way, and to explicitly produce the cofree coring and limit of corings, which we now do. We use the notations of equations \ref{e.co1} and \ref{e.co2}. We thus have 
%\begin{eqnarray*}
%C(V) & = & \lim\limits_{\stackrel{\longrightarrow}{U(G)\rightarrow V|\,G\in Crg_A; \,|G|\leq %\{|A|,\aleph_0\}}}G
%\end{eqnarray*}
%For the limit of corings, let $\Aa$ be a small category, and $F:\Aa\rightarrow Crg_A$ a functor; denote $(C_a)_a$ the objects in $F(\Aa)$. Let $P=\lim\limits_{\stackrel{\longleftarrow}{a}}C_a$ be the limit of $A$-bimodules, with canonical maps $q_a:P\rightarrow C_a$. Let $f:C(P)\rightarrow P$ be the cofree coring, and let $(C(P)_0,p_0)=(\Hh_0(P,(q_af)_a),p_0)$ be the construction of \ref{r.cat} relative to the morphisms $q_af$. In other words, 
%$$C(P)_0=\lim\limits_{\stackrel{\longrightarrow}{[h:H\rightarrow C(P)]|h,q_afh\in Crg_A}}H$$ Then $C(P)_0$ together with the maps $(q_afp_0)_a$ is the colimit of the system $(C_a)_a$ (i.e. of the functor $F$).
We also note that the system of generators for $Crg_A$ can be chosen better in some situations, as it was the case for von Neumann regular rings.

\begin{proposition}
Let $A$ be a countably generated $\KK$-algebra over a field $\KK$. If $C$ is an $A$-coring, then for every at most countable dimensional $A$-sub-bimodule $M$ of $C$, there is a subcoring $D$ of $C$ such that $D$ is a pure left and right $A$-submodule of $C$, $M\subseteq D$ and $\dim(D)\leq \aleph_0$. In particular, the $A$-corings of at most countable dimension generate $Crg_A$.
\end{proposition}
\begin{proof}
Under the current hypotheses and previous notations, note that if $N\subset C$ with $\dim(N)\leq \aleph_0$, we may choose the submodule $N^*$ such that $\dim(N^*)\leq \aleph_0$. Indeed, one needs to recursively add all solutions to the system of equations \ref{eq1}; but at each step, it is enough to consider all such systems with $m_{ij}$-coefficients in a fixed basis of $N^{(k)}$, and with coefficients $\lambda_{ij}\in A$ coming from from a fixed basis of $A$. Iterating this countably many times, each $N^{(k)}$ has $\dim(N^{(k)})\leq \aleph_0$ and so $\dim(N^*)\leq \aleph_0$. Similarly, given $M$ with $\dim(M)\leq \aleph_0$, for a basis $(m_i)_{i\geq 1}$ choose representations as finite tensors $\Delta(m_i)=\sum\limits_{j} m_{ij} \otimes n_{ij}$ and let $\bar{M}$ be the $A$-sub-bimodule generated by all the $m_i,m_{ij},n_{ij}$; it will be at most countable dimensional since $\dim(A)\leq \aleph_0$. Hence, the $M^\sim$ constructed before can be chosen with $\dim(M^\sim)\leq \aleph_0$, and countable iteration combining the two processes yields a countable dimensional pure subcoring.
%Iterating this and taking the countable union will again yield  countable dimensional $A$-submodule $M^\sim$ constructed as above. Finally, the countable iteration combining the two processes yields a countable dimensional pure subcoring as in the statement.
\end{proof}

In the above case, we get $C(V)=\lim\limits_{\stackrel{\longrightarrow}{\{D\rightarrow V\}}}D$ the colimit ranging over the category of linear maps $D\rightarrow V$ from at most countable dimensional corings $D$ to $V$. Similarly, one has

\begin{corollary}
If $A$ is a $\KK$ algebra with $\dim(A)=\gamma$ and $C$ is an $A$-coring, then every subset $F$ of $C$ with $|F|\leq \gamma$ is contained in a subcoring $D$ of $C$ which is a pure left and right $A$-submodule of $C$ and such that $\dim(D)\leq \max\{\gamma,\aleph_0\}$.
\end{corollary}

In the case of $CoAlg_\KK$ = coalgebras over a field $\KK$, we note that the generators are much more precise. Of course, the finite dimensional coalgebras generate $CoAlg_\KK$. This explains one of the classical constructions of the cofree coalgebra $C(V)$ over a vector space $V$ (see \cite[Chapter 1]{Dascalescu}) in a new perspective. For example, if $V$ is finite dimensional, $C(V)=T(V^*)^0$, where $A^0$ denotes the finite dual coalgebra of the algebra $A$; one has $A^0=\lim\limits_{I{\rm\,cofinite\,ideal}}(A/I)^*$. But for a finite dimensional algebra $B$, $\Hom_{Alg}(T(V^*),B)=\Hom_\KK(V^*,B)=\Hom_\KK(B^*,V)$.
%But algebra maps $T(V^*)\rightarrow B$ for a finite dimensional algebra $B$ are in 1-1 correspondence with linear maps $V^*\rightarrow B$, or $\Hom_\KK(B^*,V)$, and $B^*$ is a finite dimensional coalgebra. 
So, the colimit building $T(V^*)^0$ is thus the colimit of $D$ over linear maps $D\rightarrow V$ for finite dimensional coalgebras $D$. Recall that if $n\geq 1$, the coalgebra $M_n^c(\KK)=(M_n(\KK))^*$ is called the comatrix coalgebra over $\KK$.

\begin{proposition}
The set $(M_n^c(\KK))_n$ generates the category $CoAlg_\KK$of $\KK$-coalgebras over a field $\KK$.
\end{proposition}
\begin{proof}
Obvious, as any finite dimensional coalgebra has a comatrix basis, i.e. it is a quotient of some $M_n^c(\KK)$ (since finite dimensional algebras embed in matrix algebras).
\end{proof}

\section{Cofree coalgebras in monoidal categories of modules and comodules over bialgebras and Hopf algebras}\label{s.3}

As an application of the above ideas, we also look at the case of coalgebras in the category $\Mm^H$ of right $H$-comodules and ${}_H\Mm$ of left $H$-modules, where $H$ is a bialgebra (or a Hopf algebra). We refer the reader to \cite{Dascalescu} and \cite{Radford} for basic Hopf algebra definitions and notations. Let $CoAlg(\Mm^H)$ and $CoAlg({}_H\Mm)$ be the category of coalgebras in $\Mm^H$ and ${}_H\Mm$, respectively. %Recall that ${}_H\Mm$ is monoidal with the tensor product of two left $H$-modules $V,W$ being $V\otimes W=V\otimes_\KK W$, and action $h(v\otimes w)=\sum h_1v\otimes h_2w$ using sigma notation. The tensor product in $\Mm^H$ of two comodules $V,W$ is given, in sigma notation, by $(v\otimes w)_0\otimes (v\otimes w)_1=\sum v_0\otimes w_0\otimes v_1w_1$. $\KK$ is a the trivial $H$-module via $\varepsilon$ and $H$-comodule via the unit map $\KK\rightarrow H$, and is the unit object in the two categories, respetively.
Recall that $\Mm^H$ and ${}_H\Mm$ are monoidal categories; moreover, the objects of the category $CoAlg(\Mm^H)$ are precisely the $H$-comodule coalgebras, and the objects of $CoAlg({}_H\Mm)$ are the $H$-module coalgebras (see \cite[Chapter 6]{Dascalescu}).
%
%\vspace{.2cm}
%
Using again Remark \ref{r.cat} and Proposition \ref{p.cowp} and the remarks in section \ref{s.p} we obtain:

\begin{proposition}
The categories of coalgebras $CoAlg(\Mm^H)$ and $CoAlg({}_H\Mm)$ are cocomplete, co-wellpowered, and the forgetful functors $F^H:CoAlg(\Mm^H)\longrightarrow \Mm^H$ and $F_H:CoAlg({}_H\Mm)\longrightarrow{}_H\Mm$ preserve colimits.
\end{proposition}

We explicitly write generators for these two categories. We will have somewhat different situations. 

\subsection*{Module coalgebras}

\begin{proposition}
Let $C$ be a left $H$-module coalgebra, and $F$ a finite subset of $C$. Then $F$ is contained in a submodule subcoalgebra $D$ of $C$ (i.e. $D$ is a $H$-submodule and subcoalgebra of $C$) which is finitely generated as a left $H$-module .
\end{proposition}
\begin{proof}
Let $E\subseteq C$ be the subcoalgebra generated by $F$; it is finite dimensional by the fundamental theorem of coalgebras. Let $D=HE$ be the left $H$-submodule of $C$ generated by $E$. Write $\Delta_H(h)=\sum h_1\otimes h_2$ and $\Delta_C(c)=\sum c_{(1)}\otimes c_{(2)}$. If $c\in E, h\in H$, then $\sum(h\cdot c)_{(1)}\otimes (h\cdot c)_{(2)}=h_1c_{(1)}\otimes h_2c_{(2)}\in H\cdot E\otimes H\cdot E$, %(since $\sum c_{(1)}\otimes c_{(2)}$ can be chosen with a representation in which both components  are in $E$), 
so $H\cdot E$ is a $H$-submodule subcoalgebra of $C$. %It is easy to see that the inclusion $i:H\cdot E\hookrightarrow C$ is then a morphism of $H$-modules and coalgebras, and becomes a coalgebra morphism in the category $Coalg({}_H\Mm)$.
\end{proof}

\begin{corollary}
The left $H$-module coalgebras $f.g.CoAlg({}_H\Mm)$ which are finitely generated as left $H$-modules form a system of generators for $CoAlg({}_H\Mm)$. Consequently, the functor $F_H:CoAlg({}_H\Mm)\rightarrow {}_H\Mm$ has a right adjoint $G_H$ given by 
$$G_H(V)=\lim\limits_{\stackrel{\longrightarrow}{[f:D\rightarrow V]\in{}_H\Mm,\,D\in CoAlg({}_H\Mm)},\,{}_HD-f.g.}D.$$ 
The limit of a system of $H$-module coalgebras (functor on a small category) $F:\Aa\rightarrow CoAlg({}_H\Mm)$ is $(\Hh_0(\lim\limits_{\stackrel{\longleftarrow}{\Aa}}F(a),(q_a)_a),p_0)$, where $q_a$ are the canonical maps of the limit of $H$-modules $\lim\limits_{\stackrel{\longleftarrow}{\Aa}}F(a)=\lim\limits_{\stackrel{\longleftarrow}{\Aa}}F_H\circ F(a)$.
%, and hence the fogetful functor from the category of left $H$-module coalgebras to the left $H$-modules has a right adjoint, the co-free $H$-module coalgebra on an $H$-module. 
\end{corollary}

%Note that if $M$ is a left $H$-module, then $G_H(M)=\lim\limits_{\stackrel{\longrightarrow}{D\rightarrow M}} D$ is computed as the colimit over the (small) category of morphims of $H$-modules $D\rightarrow M$ where $D$ is a left $H$-module coalgebra which is finitely generated over $H$. 

\subsection*{Comodule coalgebras}

We can prove a stronger statement for this case. Recall that for an algebra $A$, the finite dual (or representative) coalgebra of $A$ is $A^0=\{f\in A^* \vert \exists\,n,\, g_i,h_i\in A^*,\,i=1,\dots,n\,\,{\rm such\,that\,} f(ab)=\sum\limits_{i=1}^ng_i(a)h_i(b),\,\forall\, a,b\in A\}$. We refer to \cite{Dascalescu} for equivalent characterizations. %, equivalently, $A^0$ contains all $f\in A^*$ whose kernel contains an ideal of $A$ of finite codimension. This is also called the coalgebra of representative functions, and is spanned by coefficients of all finite dimensional representations of $A$. 
We note that while in general $A^0\neq A^*$ unless $\dim(A)<\infty$, in fact every $f\in A^*$ is ``locally representative" in the sense of the following Lemma, which is standard linear algebra.

%Recall that for an algebra $A$, the finite dual of $A$ is a coalgebra $A^0$ which consists of all $f\in A^*$ for which there are finite families $g_i,h_i,\,i=1,\dots,n$ ($n$ depends on $f$) such that $f(ab)=\sum\limits_{i=1}^ng_i(a)h_i(b)$ for all $a,b\in A$, equivalently, $A^0$ contains all $f\in A^*$ whose kernel contains an ideal of $A$ of finite codimension. This is also called the coalgebra of representative functions, and is spanned by coefficients of all finite dimensional representations of $A$. We note that while in general $A^0\neq A^*$ unless $A$ is finite dimensional, in fact every $f\in A^*$ is ``locally representative" in the following sense:

\begin{lemma}
Let $A$ be a $\KK$-algebra and $f\in A^*$; then for every finite dimensional $V\subseteq A$, there are families $g_i,h_i,\,i=1,\dots,n$ for which $f(ab)=\sum\limits_{i=1}^ng_i(a)h_i(b)$ for all $a,b\in V$.
\end{lemma}
\begin{proof}
$(V\otimes V)^*\cong V^*\otimes V^*$ if $\dim(V)<\infty$, so every linear $\varphi:V\otimes V\rightarrow \KK$ is of the form $\varphi(a\otimes b)=\sum\limits_ig_i(a)h_i(b)$; if $i:V\hookrightarrow A$ is the inclusion and $m:A\otimes A\rightarrow A$ the multiplication, the statement follows by using the map $A^*\stackrel{m^*}{\longrightarrow}(A\otimes A)^*\stackrel{(i\otimes i)^*}{\longrightarrow}(V\otimes V)^*\cong V^*\otimes V^*$ and noting that $f\vert_{V\otimes V}=(i\otimes i)^*m^*(f)$ is in $V^*\otimes V^*$ of the desired form.
%if $i:V\hookrightarrow A$ is the inclusion and $m:A\otimes A\rightarrow A$ the multiplication, the map $A^*\stackrel{m^*}{\longrightarrow}(A\otimes A)^*\stackrel{i\otimes i}{\longrightarrow}(V\otimes V)^*\cong V^*\otimes V^*$ shows that $m^*(f)$ restricted to $V\otimes V$ is of the desired form as an element ``in"  $V^*\otimes V^*$.
\end{proof}

The following finiteness theorem may be known, but we were unable to locate a reference. Its proof is  straightforward, and we sketch it for completeness. We recall that if $C$ is a coalgebra and $(M,\rho:M\rightarrow M\otimes C)$ is a finite dimensional right $C$-comodule, then there is a finite dimensional subcoalgebra $D$ of $C$ such that $\rho(M)\subseteq M\otimes D$. The smallest such coalgebra $cf_C(M)$ is unique and is called the coalgebra of coefficients of $M$ (see \cite[Chapter 1]{Dascalescu}). 

\begin{theorem}[Finiteness theorem for comodule coalgebras]
Let $H$ be a bialgebra and $C$ be a right $H$-comodule coalgebra. Then every finite subset $F$ of $C$ is contained in a finite dimensional subcomodule subcoalgebra $D$ of $C$.%; that is, $D$ is a finite dimenisonal vector subspace of $C$ and it is an $H$-subcomodule and a subcoalgebra of $C$ (and consequently, the inclusion map $D\hookrightarrow C$ is a morphism of comodule coalgebras.)
\end{theorem}
\begin{proof}
This can be proved by applying the finiteness theorem of comodules over coalgebras, using the crossed coproduct coalgebra $H\mathrel\triangleright\joinrel\mathrel< C$ (see \cite{Molnar}) and regarding an $H$-subcomodule subcoalgebra $D$ of $C$ as a subcomodule of $C$ over the coalgebra $(H\mathrel\triangleright\joinrel\mathrel< C)\otimes C^{cop}$. However, a direct proof can be done as well. %\\ 
%One way to see this is to consider a crossed coproduct coalgebra $H\mathrel\triangleright\joinrel\mathrel< C$ on $H\otimes C$ (\cite{Molnar}); a subcomodule subcoalgebra $D$ is a $H$-subcomodule and a left and right $C$-subcomodule of $C$, equivalently, a subcomodule of $C$ over a certain coalgebra $(H\mathrel\triangleright\joinrel\mathrel< C)\otimes C^{cop}$; the theorem of finite dimensional comodules over coalgebras applies to yield the result. \\
Let $E$ be a finite dimensional subcoalgebra of $C$ with $F\subseteq E$. Then $D=H^*\cdot E$ is the right $H$-subcomodule of $C$ generated by $E$ and $\dim(D)<\infty$. It is enough to show that $D$ is a subcoalgebra of $C$. Let $V=cf_H(D)\subseteq H$ ($H$-coefficients of $D$), so $\dim(V)<\infty$ as $\dim(D)<\infty$. For $\alpha\in H^*$, let $\beta_i,\gamma_i,\,i=1,\dots,n\in H^*$ be such that $\alpha(ab)=\beta_i(a)\gamma_i(b)$ for $a,b\in V$. Denote $\rho_C(c)= c_0\otimes c_1$ the $H$-comodule structure of $C$ and $\Delta_H(h)= h_1\otimes h_2$, $\Delta_C(c)=c_{(1)}\otimes c_{(2)}$ as before with omitted summation symbol. Then for $d\in E$
\begin{eqnarray*}
(\alpha\cdot d)_{(1)}\otimes (\alpha\cdot d)_{(2)} & = & (\alpha(d_1)d_0)_{(1)} \otimes (\alpha(d_1)d_0)_{(2)} = \alpha(d_1)d_{0(1)}\otimes d_{0(2)} \\
& = & \alpha(d_{(1)1}d_{(2)1})d_{(1)0}\otimes d_{(2)0} \,\,\,\,\,\,\, {\rm (comodule\,\,coalgebra\,\,axiom)}\\
& = & \sum\limits_i \beta_i(d_{(1)1})\gamma_i(d_{(2)1})d_{(1)0}\otimes d_{(2)0} \\
&  & \,\,\,\,\,\,\, {\rm \,since\,\,the\,\,\,} d_{(1)1},d_{(1)2} {\rm's\,\,\,are\,\,in\,\,}V=cf_H(D)\\
& = & \sum\limits_i\beta_i\cdot d_{(1)}\otimes \gamma_i\cdot d_{(2)}\in H^*\cdot D\otimes H^*\cdot D
\end{eqnarray*}
%\begin{eqnarray*}
%(\alpha\cdot d)_{(1)}\otimes (\alpha\cdot d)_{(2)} & = & (\alpha(d_1)d_0)_{(1)} \otimes %(\alpha(d_1)d_0)_{(2)} \\
%& = & \alpha(d_1)d_{0(1)}\otimes d_{0(2)} \\
%& = & \alpha(d_{(1)1}d_{(2)1})d_{(1)0}\otimes d_{(2)0} \,\,\,\,\,\,\, {\rm \,by\,\ref{eq.cc1}}\\
%& = & \sum\limits_i \beta_i(d_{(1)1})\gamma_i(d_{(2)1})d_{(1)0}\otimes d_{(2)0} \\
%& = & \,\,\,\,\,\,\, {\rm \,since\,\,the\,\,\,} d_{(1)1},d_{(1)2} {\rm's\,\,\,are\,\,in\,\,}V=cf_H(D)\\
%& = & \sum\limits_i\beta_i\cdot d_{(1)}\otimes \gamma_i\cdot d_{(2)}\in H^*\cdot D\otimes H^*\cdot D
%\end{eqnarray*}
%This ends the proof.
\end{proof}

%This imediately yields 

\begin{corollary}
The category $CoAlg(\Mm^H)$ (=right $H$-comodule coalgebras) is generated by objects which are finite dimensional. Consequently, $F^H$ has a right adjoint $G^H$ given by %The cofree $H$-comodule coalgebra over a right $H$-comodule $V$ is:% the colimit of the functor \\
$$G^H(V)=\lim\limits_{\stackrel{\longrightarrow}{[f:D\rightarrow V]\in \Mm^H,\,D\in fin.dim.CoAlg(\Mm^H)}}D.$$
The limit of a system of $H$-comodule coalgebras $F:\Aa\rightarrow CoAlg(\Mm^H)$ is $(\Hh_0(\lim\limits_{\stackrel{\longleftarrow}{\Aa}}F(a),(q_a)_a),p_0),$ where $q_a$ are the canonical maps of the limit of $H$-comodules $\lim\limits_{\stackrel{\longleftarrow}{\Aa}}F(a)=\lim\limits_{\stackrel{\longleftarrow}{\Aa}}F^H\circ F(a)$.
%$\{[f:D\rightarrow V]\in \Mm^H, D\textrm{ f.d.comodule coalgebra,}\rightarrow \Mm^H$, $[D\stackrel{f}{\longrightarrow} V]\longmapsto D$.%, where the first set consists of morphisms of finite dimensional $H$-comodules from a finite dimensional $H$-comodule coalgebra $D$.
\end{corollary}

If $H$ is a Hopf algebra with antipode $S$, the category $CoAlg(\Mm^H)$ has special systems of matrix-like generators. Recall that if $M$ is a finite dimensional $H$-comodule, then the dual space $M^*$ has a structure of a right $H$-comodule, and $\Mm^H$ has right dual objects (we refer to \cite[Chapter 2]{Bakalov} for right and left duals in monoidal categories). If $(v_i)_i; (v_i^*)_i$ are finite dual bases of $M$, then the maps $co_M:\KK\rightarrow M\otimes M^*$, $co_M(1)=\sum\limits_{i}v_i\otimes v_i^*$, and $ev_M:M^*\otimes M\rightarrow \KK$, $ev_M(v^*\otimes v)=v^*(v)$ are called co-evaluation and evaluation respectively. 
%Recall that if $M$ is a finite dimensional $H$-comodule, equivalently, a left rational $H^*$-module, then $M^*$ has a structure of a right $H^*$-module which is rational, it is a left $H$-comodule. Using the antimorphism of coalgebras $S:H\rightarrow H$, $M^*$ becomes a right $H$-comodule, called the right dual of $M$. The maps $co_M:\KK\rightarrow M\otimes M^*$ defined by $co_M(1)=\sum\limits_{i}v_i\otimes v_i^*$, where $(v_i)_i; (v_i^*)_i$ are finite dual bases of $M$, and $ev_M:M^*\otimes M\rightarrow \KK$ defined by $ev_M(v^*\otimes v)=v^*(v)$ are called co-evaluation and evaluation respectively. They are morphisms of right $H$-comodules; we refer to \cite[Chapter 2]{Bakalov} for right and left duals in monoidal categories. 
%\vspace{.2cm}
Recall also that the right $H$-comodule $V\otimes V^*$ becomes an algebra in $\Mm^H$ with multiplication given by the canonical map $1_V \otimes ev_V\otimes 1_{V^*}:V\otimes (V^*\otimes V)\otimes V^*\cong (V\otimes V^*)\otimes (V\otimes V^*)\longrightarrow V\otimes \KK\otimes V^*\cong (V\otimes V^*)$ and unit $co_V:\KK\rightarrow V\otimes V^*$. Similarly, the right $H$-comodule $V^*\otimes V$ becomes a coalgebra in $\Mm^H$ via the analogous dual maps.

\vspace{.2cm}

We recall a few well known facts to prove the following theorem on generators in $CoAlg(\Mm^H)$. If $U,V,W\in \Mm^H$, then $\Hom^H(U\otimes V,W)=\Hom^H(U,W\otimes V^*)$, i.e. $V\otimes W^*$ is the internal Hom from $W$ to $V$ in $\Mm^H$ \cite{Bakalov}. In particular, if $(A,m,u)$ is an algebra in $\Mm^H$, $\Hom^H(A\otimes A,A)=\Hom^H(A,A\otimes A^*)$. If $(v_i)_{i=1,\dots,n}, (v_i^*)_{i=1,\dots,n},$ is a dual basis of $A$, then this isomorphism takes $m:A\otimes A\rightarrow A$ to $\varphi(m)=(a\longmapsto \sum\limits_iav_i\otimes v_i^*)$. Of course, as a $\KK$-algebra, $A\otimes A^*$ is simply isomorphic to $M_n(\KK)$; hence, if we forget the $H$-comodule structures, if $m:A\otimes A\rightarrow A$ is the multiplication of $A$, then $\varphi(m):A\rightarrow A\otimes A^*$ is a morphism of algebras corresponding to the left regular representation of $A$. By the canonical isomorphism $\Hom^H(U\otimes V,W)=\Hom^H(U,W\otimes V^*)$, $\varphi(m)$ is a morphism in $\Mm^H$ if $m$ is so, and therefore, for a $H$-comodule algebra $A$, $\varphi(m):A\rightarrow A\otimes A^*$ is a morphism of $H$-comodule algebras. Also, $\varphi(m)$ is obviously injective, so it is a monomorphism in $Alg(\Mm^H)$. Hence, we have

\begin{theorem}
The finite dimensional algebras of the form $V\otimes V^*$ for finite dimensional $H$-comodules $V$, form a system of cogenerators in the category $fdAlg(\Mm^H)$ of finite dimensional algebras in $\Mm^H$ (and also in $Alg(\Mm^H)$). The coalgebras $V^*\otimes V$ form a system of generators of $CoAlg(\Mm^H)$ (= the category of $H$-comodule coalgebras). 
\end{theorem} 

We end by noting that besides using the above generators, a construction for cofree coalgebras in $\Mm^H$ similar to that over vector spaces can be done. % As closing remarks, one can see that in the case of $\Mm^H$ for a Hopf algebra $H$, besides using the above systems of generators for the description of the cofree coalgebra on an $H$-comodule $V$, one can also do a construction similar to that of vector spaces. 
Namely, for an algebra $A$ in $\Mm^H$ one can define a right finite dual $A^0$ of $A$ in a natural way. If the antipode of $H$ is bijective, one can also define a left finite dual ${}^0A$ using left duals. If $V$ is a finite dimensional right $H$-comodule, then the tensor algebra $T({}^*V)$ in the category $\Mm^H$ has a right finite dual $T({}^*V)^0$ which is a coalgebra in $\Mm^H$, and it is the cofree coalgebra on $V$. This can be extended to arbitrary $V$'s as in the case of coalgebras over fields. We leave the details of this construction to the reader. In particular, these considerations apply to the category $YD(H)$ of Yetter-Drinfeld modules over $H$ as well.

%\begin{corollary}
%If $H$ is a Hopf algebra and $YD(H)$ is the category of Yetter-Drinfeld modules over $H$, then there is a cofree coalgebra on any object in $YD(H)$.
%\end{corollary}
%\begin{proof}
%It follows because $YD(H)$ is equivalent with the category of modules over the Hopf algebra $D(H)$, the Drinfeld double of $H$.
%\end{proof}

%We further note that the above constructions also yield the existence of a cofree coalgebra in other categories, such as 

%We also ask whether a cofree coalgebra can be constructed in general monoidal categories, and even in the case when the category is not abelian. The conditions below seem natural to ask, perhaps along other additional ones, and the answer to the following question, seems to only depend on the ability of constructing generators for the category of coalgebras in $\Cc$.

%\begin{question}
%Let $(\Cc,\otimes,I)$ be a monoidal category, which is assumed additive. Assume that :\\
%$\bullet$ $\Cc$ is cocomplete and the tensor product commutes with colimits on both variables\\
%$\bullet$ There is a set of generators for $\Cc$ and \\
%$\bullet$ $\Cc$ is co-wellpowered.\\
%Let $Coalg(\Cc)$ be the category of coalgebras (or comonoids) in $\Cc$. Does the forgetful functor $Coalg(\Cc)\rightarrow \Cc$ have a right adjoint? 
%\end{question}

%Proposition \ref{p.gen} Proposition \ref{p.cowp} Lemma \ref{l.surj}

\vspace*{3mm} 
\begin{flushright}
\begin{minipage}{148mm}\sc\footnotesize

Adnan Hashim Abdulwahid\\
University of Iowa, \\
Department of Mathematics, MacLean Hall\\
Iowa City, IA, USA\\
{\it E--mail address}: {\tt
adnan-al-khafaji@uiowa.edu}\vspace*{3mm}

Miodrag Cristian Iovanov\\
University of Iowa\\
Department of Mathematics, MacLean Hall \\
Iowa City, IA, USA\\
%Miodrag Cristian Iovanov\\
%University of Southern California\\
%Department of Mathematics, 3620 South Vermont Ave. KAP 400C \\
%Los Angeles, California 90089-2532\\
%and\\
%University of Bucharest, Faculty of Mathematics, Str. Academiei 14\\ 
%RO-010014, Bucharest, Romania\\
%State University of New York (Buffalo)\\
%Department of Mathematics, 244 Mathematics Building\\
%Buffalo, NY 14260-2900, USA\\
{\it E--mail address}: {\tt
miodrag-iovanov@uiowa.edu}\vspace*{3mm}

\end{minipage}
\end{flushright}


\begin{thebibliography}{20}

\bibitem{Agore} A.L. Agore. {\it Limits of Coalgebras, Bialgebras and Hopf Algebras}, Proc. Amer. Math. Soc. 139, 855-863, 2011.

\bibitem{Bakalov}
B. Bakalov, A. Kirillov Jr., {\it Lectures on Tensor Categories and Modular Functor}, AMS University Lecture Series, 2001, 221 pp.

\bibitem{Barr} M. Barr, {\it Coalgebras Over a Commutative Ring}, J. Algebra 32, 600-610, 1974.

\bibitem{Block} R.E. Block, P. Leroux, {\it Generalized Dual Coalgebras of Algebras, with Applications to Cofree Coalgebras}, J. Pure Appl. Algebra 36, 15--21, 1985.

%\bibitem{Borceux} F. Borceux, {\it Handbook of Categorical Algebra 1: Basic Category Theory}, Cambridge University Press. 1994.

\bibitem{Brzezinski} T. Brzezinski, R. Wisbauer, {\it Corings and comodules}, London Mathematical Society Lecture Note Series Vol 309, Cambridge University Press, 2003.

\bibitem{Cohn} P.M. Cohn, {\it On the Free Product of Associative Rings}, Math. Z., 71, 380--398, 1959.
  
\bibitem{Dascalescu} S. Dascalescu, C. Nastasescu, S. Raianu, {\it Hopf Algebras: An Introduction.}, 
Marcel Dekker, Inc. 2001.

\bibitem{Fox} T.F. Fox, {\it The Construction of Cofree Coalgebras}, J. Pure Appl. Algebra 84, 1993.
  
%\bibitem{Freyd} P.J. Freyd, A. Scedrov, {\it Categories, Allegories}, Elsevier Science Publishing Company, Inc. 1990.

\bibitem{Hazewinkel} M. Hazewinkel, {\it Cofree Coalgebras and Multivariable Recursiveness}, J. Pure Appl. Algebra 183, 61 -- 103, 2003.

\bibitem{Mac Lane}
S. Mac Lane, {\it Categories for the Working Mathematician}, Graduate Texts in Mathematics vol. 5, 2nd ed., Springer, 1978. 
 
%\bibitem{Mitchell}  B. Mitchell, {\it Theory of Categories},  Academic Press, 1965.
 
\bibitem{Molnar} R.K. Molnar, {\it Semi-direct products of Hopf algebras}, J. Algebra 47 (1977), no. 1, 29 - 51.
 
%\bibitem{Pareigis} B. Pareigis, {\it Categories and Functors}, Academic Press, 1971.

\bibitem{Porst1} H.-E. Porst, \emph{On corings and comodules}, Arch. Math. (Brno) 42 (2006), no. 4, 419-425. 

\bibitem{Porst2} H.-E. Porst, \emph{Fundamental constructions for coalgebras, corings, and comodules}, Appl. Categ. Structures 16 (2008), no. 1-2, 223-238. 

\bibitem{Radford} D.E. Radford, {\it Hopf Algebras}, World Scientific Publishing, 2012.

%\bibitem{Schubert} H. Schubert, {\it Categories}, Springer-Verlag, 1972.

\bibitem{Wisbauer} R. Wisbauer, {\it Foundations of Module and Ring Theory: A Handbook for Study and Research.}, Springer-Verlag. 1991.

\end{thebibliography}
\end{document}